\documentclass[reqno]{amsart}

\usepackage{amsmath, amssymb, amsfonts}
\usepackage{mathrsfs}
\usepackage{hyperref}
\usepackage[pdftex]{color}
\usepackage{graphicx}

\newtheorem{theorem}{Theorem}[section]
\newtheorem{corollary}[theorem]{Corollary}
\newtheorem{lemma}[theorem]{Lemma}
\newtheorem{proposition}[theorem]{Proposition}
\theoremstyle{remark}
\newtheorem{remark}{Remark}[section]
\theoremstyle{definition}
\newtheorem{define}[theorem]{Definition}

\newtheorem*{problem}{Problem}
\newtheorem*{Acknowledgement}{Acknowledgement}

\newcommand{\ve}{\varepsilon}
\newcommand{\nhd}{neighborhood}
\newcommand{\va}{\varphi}
\newcommand{\fr}{\frac}	
\newcommand{\psh}{plurisubharmonic}
\newcommand{\B}{\mathbb B}
\newcommand{\re}{\mathrm{Re}}
\newcommand{\n}{\noindent}

\begin{document}

\title[S\MakeLowercase{ome properties of $h$-extendible domains}]{\Large S\MakeLowercase{ome properties of $h$-extendible domains in} $\mathbb C^{n+1}$}
\author{Ninh Van Thu\textit{$^{1,2}$} and Nguyen Quang Dieu\textit{$^{3,2}$} } 
%\thanks{The research of the author was supported VNU}
\address{Ninh Van Thu}
\address{\textit{$^{1}$}~Department of Mathematics,  VNU University of Science, Vietnam National University, Hanoi, 334 Nguyen Trai, Thanh Xuan, Hanoi, Vietnam}
\address{\textit{$^{2}$}~Thang Long Institute of Mathematics and Applied Sciences,
	Nghiem Xuan Yem, Hoang Mai, HaNoi, Vietnam}
\email{thunv@vnu.edu.vn}

\address{Nguyen Quang Dieu}
\address{\textit{$^{3}$}~Department of Mathematics, Hanoi National University of Education, 136 Xuan Thuy, Cau Giay, Hanoi, Vietnam}
 \address{\textit{$^{2}$}~Thang Long Institute of Mathematics and Applied Sciences,
Nghiem Xuan Yem, Hoang Mai, HaNoi, Vietnam}
\email{ngquang.dieu@hnue.edu.vn}
  
 \maketitle  
\begin{abstract}
The purpose of this article is twofold. The first aim is to characterize $h$-extendibility of smoothly bounded pseudoconvex domains in $\mathbb C^{n+1}$ by their noncompact automorphism groups. Our second goal is to show that if the squeezing function tends to $1$ at an $h$-extendible boundary point of a smooth pseudoconvex domain in $\mathbb C^{n+1}$, then this point must be strongly pseudoconvex.  
\end{abstract}

%\tableofcontents

\section{Introduction}  

Let $\Omega$ be a domain in $\mathbb C^n$ and let us denote by $\mathrm{Aut}(\Omega)$ the group of biholomorphic self-maps of $\Omega$ with the compact-open topology. It is proved by H. Cartan (see \cite{Nar71}) that if $\Omega$ is a bounded domain in $\mathbb C^n$ and the $\mathrm{Aut}(\Omega)$ is noncompact then there exist a point $x \in \Omega$, a point $p \in \partial \Omega$, and automorphisms $\varphi_j \in \mathrm{Aut}(\Omega)$ such that $\varphi_j(x) \to p$. In this circumstance, we call $p$ a {\it boundary orbit accumulation point}.
Moreover, if $\partial \Omega$ enjoys some sort of convexity at $p$ then $\varphi_j$ converges uniformly on compact sets of $\Omega$ to $p$.

It is known that the local geometry of the so-called ``boundary orbit accumulation point'' $p$ in turn gives global information about the characterization of model of the domain. We refer the reader to the recent survey \cite{IK} and the references therein for the development in related subjects. For instance, B. Wong and J. P. Rosay (see \cite{W}, \cite{R}) proved the following remarkable theorem.

\n
{\bf Theorem (Wong-Rosay).} {\it Any bounded domain $\Omega \Subset\mathbb C^n$ with a $C^2$ strongly pseudoconvex boundary orbit accumulation point is biholomorphic to the unit ball in $ \mathbb C^n$.  }

After that, by using the scaling technique, introduced by S. Pinchuk \cite {P}, E. Bedford and S. Pinchuk \cite{B-P2}, F. Berteloot \cite{Ber94} proved several results about the characterization of the complex ellipsoids and models. In \cite{DN09}, Do Duc Thai and the first author showed that if $\Omega$ is  pseudoconvex finite type and smooth of class $\mathcal C^\infty$ in some neighborhood of a boundary orbit accumulation point, $\xi_0 \in \partial \Omega$, and the Levi form has corank at most one at $\xi_0$, then $\Omega$ is biholomorphically equivalent to a model
$$
M_H=\{(z_1,\cdots,z_n,w)\in \mathbb  C^n\times \mathbb C: \mathrm{Re}(w)+H(z_1,\bar z_1)+\sum_{k=1}^{n}|z_k|^2<0\},
$$
where $H$ is a homogeneous subharmonic polynomial with $\Delta H\not\equiv 0$.

To give a statement of our result, we recall that a smooth pseudoconvex boundary point $p\in \partial\Omega$ is called $h$-extendible \cite{Yu94, Yu95} (or semiregular \cite{DH94}) if Catlin's multitype and D'Angelo multitype at $p$ coincide. It is well-known that the class of  $h$-extendible points includes pseudoconvex finite points in $\mathbb C^2$, strongly pseudoconvex points in $\mathbb C^n$, and convex finite type points $\mathbb C^n$. In particular, any pseudoconvex finite type boundary point in $\mathbb C^n$ with corank of the Levi form at most one is $h$-extendible. 

The first aim in this paper is to prove the following theorem, which gives a characterization of $h$-extendible domains with noncompact automorphism groups. 
\begin{theorem}\label{maintheorem1}
Assume that $\Omega$ is a pseudoconvex domain in $\mathbb C^{n+1}$ with $\mathcal{C}^\infty$-smooth boundary $\partial \Omega$. Let $\xi_0\in\partial \Omega$ be $h$-extendible with Catlin's finite multitype $(1,m_1,\ldots,m_n)$ and let $\Lambda=(1/m_1, \ldots, 1/m_n)$. Suppose that there exists a sequence $\{\varphi_j\}\subset \mathrm{Aut}(\Omega)$ such that $\eta_j:=\varphi_j(a)$ converges $\Lambda$-nontangentially to $\xi_0$ for some $a\in \Omega$ (cf. Definition \ref{def-order}).
Then there exists a biholomorphic mapping $\sigma: \Omega \to M_P.$ Here $M_P$ is a domain of the form
$$M_{P}:=\left\{(z,w)\in \mathbb C^n\times \mathbb C\colon \mathrm{Re}(w)+P(z)<0\right \},$$
where $P$ is a $\Lambda$-homogeneous plurisubharmonic real-valued polynomial which contains no pluriharmonic monomials (cf. Definition \ref{def-28}). Moreover, the map $\sigma$ satisfies the following properties:

\n 
(a) $\sigma (a)=(0',-1)$.

\n
(b) There exist sequences $\{\xi_j\} \subset \partial \Omega$ and $\{\tilde \xi_j\} \subset \partial M_P$ such that $\xi_j \to \xi_0$ as $j\to \infty$ and that $\sigma$ extends continuously to a homeomorphism near $\xi_j$ and $\tilde \xi_j.$
\end{theorem}
\n
\begin{remark}
Recently, F. Rong and B. Zhang \cite{RZ16} gave a characterization of $h$-extendible model in which the sequence $\{\eta_j\}\subset \Omega$ converges nontangentially to an $h$-extendible boundary point $\xi_0\in \partial\Omega$. Their proof is based on the Pinchuk scaling method. However, the equation $(3.6)$ in page $905$ of \cite{RZ16}, which plays a crucial role to ensure the normality of the scaling sequence, is unclear to us. Fortunately, by using the attraction property of analytic discs based deeply on the existence of a plurisubharmonic peak function at the origin of the above model $M_P$, the normality of the scaling sequence is eventually verified (see Proposition \ref{pro-scaling}), and then the proof of Theorem \ref{maintheorem1} follows. As a consequence, the above-mentioned result of F. Rong and B. Zhang is obtained.

\n 
2. Notice that we do not know if the sequence $\{\tilde \xi_j\}$ can be chosen to be {\it bounded} even when $\partial \Omega$ is
{\it algebraic}. If this is the case then by using results in [Ber95] or [CP01] we can prove that $\sigma$ extends {\it holomorphically} through $\xi_0.$
\end{remark}

Now we move to the definition of squeezing function of a domain. Let $\Omega$ be a domain in $\mathbb C^n$ and $p \in \Omega$. For a holomorphic embedding $f\colon \Omega \to \mathbb B^n:=\B(0;1)$ with $f(p)=0$, we set
$$s_{\Omega,f}(p):=\sup\left \{r>0\colon B(0;r)\subset f(\Omega)\right\},$$
where $\mathbb B^n (z;r)\subset\mathbb{C}^n$  denotes the ball of radius $r$ with center at $z$. Then the \textit{squeezing function} $s_{\Omega}: \Omega\to\mathbb R$ is defined in \cite{DGF12} as
$$
s_{\Omega}(p):=\sup_{f} \left\{s_{\Omega,f}(p)\right\}.
$$
Note that $0 < s_{\Omega}(z)\leq 1$ for any $z \in \Omega$ and the squeezing function is clearly invariant under biholomorphic mappings.

In recent works \cite{DGF16, DFW14, KZ16} the authors proved that if $p$ is a strongly pseudoconvex boundary point, then $\lim\limits_{\Omega \ni z\to p\in \partial \Omega}s_{\Omega}(z)=1$. Conversely to this result, J. E. Forn{\ae}ss and F. E. Wold posed the following problem (see \cite[Problem $4.1$]{FW18}).
\begin{problem}
If $\Omega$ is a bounded pseudoconvex domain with smooth boundary, and if $\lim\limits_{\Omega \ni z\to p\in \partial \Omega}s_{\Omega}(z)=1$, then is the boundary of $\Omega$ strongly pseudoconvex at $p$?
\end{problem}

The main results around this problem are due to A. Zimmer \cite{Zim18a, Zim18b},  J. E. Forn{\ae}ss and F. E. Wold \cite{FW18}, S. Joo and K.-T. Kim \cite{JK18}, P. Mahajan and K. Verma \cite{MV19}. More precisely, in \cite{Zim18a, Zim18b} A. Zimmer  proved that the answer is affirmative if the domain is bounded convex with $\mathcal{C}^{2,\alpha}$-smooth boundary.  In \cite{FW18}, J. E. Forn{\ae}ss and F. E. Wold constructed a counter-example to this problem, that is, they constructed a bounded convex  $\mathcal{C}^2$-smooth domain  $\Omega\subset \mathbb C^n$ which is not strongly pseudoconvex, but 
 $$
 \lim\limits_{\Omega \ni z\to \partial \Omega}s_{\Omega}(z)=1.
 $$

Now let us consider a sequence $\{\eta_j\}\subset \Omega$ converging to an $h$-extendible boundary point $\xi_0\in \partial\Omega$. Suppose that $\Omega$ is pseudoconvex of finite type near $\xi_0$ and $\lim\limits_{j\to \infty} s_\Omega(\eta_j)=1$. It is known that if the sequence $\{\eta_j\}\subset \Omega$ converges to $\xi_0$ along the inner normal line to $\partial\Omega$ at $\xi_0$, then $\xi_0$ must be strongly pseudoconvex (see \cite{JK18} for $n=2$ and \cite{MV19} for general case). Moreover, this result was obtained in \cite{Ni18} for the case that $\{\eta_j\}\subset \Omega$ converges nontangentially to $\xi_0$. 
We are grateful to the referee for bringing to our attention the relevant reference \cite{Ni18} and for pointing out that a small modification of the methods given in \cite{Ni18} also yields Theorem \ref{theorem 1.2} below concerning the squeezing function. Nevertheless, we still believe that our method of proof has some merit since it uses the scaling technique, and therefore is quite different from \cite{Ni18}. 

The second aim in this paper is to prove the following theorem.
\begin{theorem}\label{theorem 1.2}
Let $\xi_0$ be an $h$-extendible boundary point of  a $\mathcal{C}^\infty$-smooth, bounded pseudoconvex  domain $\Omega$ in $\mathbb C^{n+1}$. Assume that $\lim\limits_{j\to \infty} s_\Omega(\eta_j)= 1$ for some sequence $\{\eta_j\}\subset \Omega$ converging $\Lambda$-nontangentially to $\xi_0$. Then $\xi_0$ is a strongly pseudoconvex point.
\end{theorem}

The organization of this paper is as follows: In Sections $2$ and $3$, we recall some basic definitions and results needed later. In Section $4$, we verify the normality of the scaling sequence and then we give a proof of Theorem \ref{maintheorem1}. Finally, the proof of Theorem \ref{theorem 1.2} is given in Section $5$.
\section{The normality of sequences of biholomorphisms}
First of all, we recall the following definition (see \cite{GK} or \cite{DN09}). 
\begin{define} Let $\{\Omega_i\}_{i=1}^\infty$ be a sequence of open sets in a complex manifold $M$ and $\Omega_0 $ be an open set of $M$. The sequence $\{\Omega_i\}_{i=1}^\infty$ is said to converge to $\Omega_0 $ (written $\lim\Omega_i=\Omega_0$) if and only if 
\begin{enumerate}
\item[(i)] For any compact set $K\subset \Omega_0,$ there is an $i_0=i_0(K)$ such that $i\geq i_0$ implies that $K\subset \Omega_i$; and 
\item[(ii)] If $K$ is a compact set which is contained in $\Omega_i$ for all sufficiently large $i,$ then  $K\subset \Omega_0$.
\end{enumerate}  
\end{define}

Next, we need the following proposition, which is a generalization of the theorem of H. Cartan (see \cite{GK, TM, DN09}).
\begin{proposition} \label{T:7}  Let $\{A_i\}_{i=1}^\infty$ and $\{\Omega_i\}_{i=1}^\infty$ be sequences of domains in a complex manifold $M$ with $\lim A_i=A_0$ and $\lim \Omega_i=\Omega_0$ for some (uniquely determined) domains $A_0$, $\Omega_0$ in $M$. Suppose that $\{f_i: A_i \to \Omega_i\} $ is a sequence of biholomorphic maps. Suppose also that the sequence $\{f_i: A_i\to M \}$ converges uniformly on compact subsets of $ A_0$ to a holomorphic map $F:A_0\to M $ and the sequence $\{g_i:=f^{-1}_i: \Omega_i\to M \}$ converges uniformly on compact subsets of $\Omega_0$ to a holomorphic map $G:\Omega_0\to M $.  Then either of the following assertions holds. 
\begin{enumerate}
\item[(i)] The sequence $\{f_i\}$ is compactly divergent, i.e., for each compact set $K\subset \Omega_0$ and each compact set $L\subset \Omega_0$, there exists an integer $i_0$ such that $f_i(K)\cap L=\emptyset$ for $i\geq i_0$; or
\item[(ii)] There exists a subsequence $\{f_{i_j}\}\subset \{f_i\}$  such that the sequence $\{f_{i_j}\}$ converges uniformly on compact subsets of $A_0$ to a biholomorphic map $F: A_0 \to \Omega_0$.
\end{enumerate}
\end{proposition}

In addition, we prepare the following proposition (see  \cite[Proposition $2.1$]{Ber94} or \cite[Proposition $2.2$]{DN09}).
\begin{proposition}\label{T:8}  Let $M$ be a domain  in a complex manifold $X$ of dimension $n$ and $\xi_0\in \partial M$. Assume that $\partial M$ is pseudoconvex and of finite type near $\xi_0$. 
\begin{enumerate}
\item[(a)] Let $\Omega $ be a domain  in a complex manifold $Y$ of dimension $m$. Then every sequence $\{\varphi_j\}\subset \mathrm{Hol}(\Omega,M)$ converges unifomly on compact subsets of $\Omega $ to $\xi_0$ if and only if $\lim   \varphi_j(a)=\xi_0$ for some $a\in \Omega$.
\item[(b)] Assume, moreover, that there exists a sequence $\{\varphi_j\}\subset Aut(M)$  such that $\lim \varphi_j(a)=\xi_0 $ for some $a\in M$. Then $M$ is taut.
\end{enumerate}
\end{proposition}  
\begin{remark} \label{r1}  By Proposition \ref{T:8} and by the hypothesis of Theorem \ref{maintheorem1}, for each compact subset $K\Subset \Omega$ and each neighborhood $U$ of $\xi_0$, there exists an integer $j_0$ such that $\varphi_j(K)\subset \Omega\cap U$ for all $j\geq j_0$. Moreover, $\Omega$ is taut. 
\end{remark}

\section{Catlin's multitype and the $h$-extendibility} 
\subsection{Catlin's multitype} 
For the convenience of the exposition, let us recall \emph{Catlin's multitype}~(for more details, we refer to \cite{Cat84, Yu92} and the references therein). Let $\Omega$ be a domain in $\mathbb C^n$ and $\rho$ be a defining function for $\Omega$ near $z_0\in \partial\Omega$. Let us denote by $\Gamma^n$ the set of all $n$-tuples of numbers $\mu=(\mu_1,\ldots,\mu_n)$ such that
\begin{itemize}
\item[(i)] $1\leq \mu_1\leq \cdots\leq\mu_n\leq +\infty$;
\item[(ii)] For each $j$, either $\mu_j=+\infty$ or there is a set of non-negative integers $k_1,\ldots,k_j$ with $k_j>0$ such that
\[
\sum_{s=1}^j \frac{k_s}{\mu_s}=1.
\]
\end{itemize}

A weight $\mu\in \Gamma^n$ is called \emph{distinguished} if there exist holomorphic coordinates $(z_1,\ldots,z_n)$ about $z_0$ with $z_0$ maps to the origin such that 
\[
D^\alpha \overline{D}^\beta \rho(z_0)=0~\text{whenever}~\sum_{i=1}^n\frac{\alpha_i+\beta_i}{\mu_i}<1.
\]
Here $D^\alpha$ and $\overline{D}^\beta$ denote the partial differential operators
\[
\frac{\partial^{|\alpha|}}{\partial z_1^{\alpha_1}\cdots \partial z_n^{\alpha_n }}~\text{and}~\frac{\partial^{|\beta|}}{\partial \bar z_1^{\beta_1}\cdots \partial \bar z_n^{\beta_n }},
\]
respectively. 
\begin{define}
The \emph{multitype} $\mathcal{M}(z_0)$ is defined to be the smallest weight $\mathcal{M}=
(m_1,\ldots,m_n)$ in $\Gamma^n$~(smallest in the lexicographic sense) such that $\mathcal{M}\geq \mu$ for every distinguished weight $\mu$.
\end{define}
 \subsection{The $h$-extendibility }
In what follows, we call a multiindex $(\lambda_1,\lambda_2, \ldots, \lambda_n)$ a \emph{multiweight} if $1\geq \lambda_1\geq \cdots\geq \lambda_n$. Now let us recall the following definitions (cf. \cite{Yu94, Yu95}).  
\begin{define} \label{def-28}Let $f(z)$ be a function on $\mathbb  C^n$ and let $\Lambda=(\lambda_1,\lambda_2, \ldots, \lambda_n)$ be a multiweight. For any real number $t\geq 0$, set
$$
\pi_t(z)=(t^{\lambda_1}z_1,t^{\lambda_2}z_2, \ldots,t^{\lambda_n}z_n).
$$ 
We say that $f$ is \emph{$\Lambda$-homogeneous with weight $\alpha$} if $f(\pi_t(z))=t^\alpha f(z)$ for every $t\geq 0$ and $z\in \mathbb C^n$. In case $\alpha=1$, then $f$ is simply called \emph{$\Lambda$-homogeneous}.
\end{define} 

For a multiweight $\Lambda$, the following function
$$
\sigma(z)=\sigma_\Lambda(z):=\sum_{j=1}^n |z_j|^{1/\lambda_j}
$$
 is $\Lambda$-homogeneous. Moreover, for a multiweight $\Lambda$ and a real-valued $\Lambda$-homogeneous function $P$, we define a homogeneous model $D_{\Lambda,P}$ as follows:
 $$
D_{\Lambda,P}=\left\{ (z,w)\in \mathbb C^n\times \mathbb C\colon \mathrm{Re}(w)+ P(z)<0 \right\}.
 $$
 \begin{define} Let $D_{\Lambda,P}$ be a homogeneous model. Then $D_{\Lambda,P}$ is called \emph{$h$-extendible} if there exists a $\Lambda$-homogeneous $\mathcal{C}^1$ function $a(z)$ on $\mathbb C^n\setminus\{0\}$ satisfying the following conditions:
 \begin{itemize}
 \item[(i)] $a(z)>0$ whenever $z\ne 0$;
 \item[(ii)] $P(z)-a(z)$ is plurisubharmonic on $\mathbb C^n$. 
  \end{itemize}
We will call $a(z)$ a \emph{bumping function}. 
 \end{define}
 \begin{remark} \label{smooth-bumping} In this paper, our model $D_{\Lambda,P}$ is always assumed to be of finite type.  So, by \cite[Theorem $2.1$]{Yu94} the bumping function $a(z)$ must be $\mathcal{C}^\infty$ on $\mathbb C^n\setminus \{0\}$ and $P(z)-a(z)$ is strictly plurisubharmonic on $\mathbb C^n\setminus \{0\}$. Moreover, $\Lambda=(1/m_1,\ldots,1/m_n)$, where $(1, m_1,\ldots, m_n)$ is the multitype of  $D_{\Lambda,P}$ at $0$. For several equivalent conditions to the $h$-extendibility, we refer the reader to \cite{Yu94}.
 \end{remark}
 
 \begin{remark}
Let $a(z)$ be a bumping function. Then there is a constant $C>0$ such that
 $$
 C \sigma(z)\leq a(z)\leq C^{-1} \sigma(z),~\forall~z\in \mathbb C^n. 
 $$
 \end{remark}

By a pointed domain $(\Omega,p)$ in $\mathbb C^{n+1}$ we mean that $\Omega$ is a smooth pseudoconvex domain in $\mathbb C^{n+1}$ with $p\in \partial\Omega$. Let $\rho$ be a local defining function for $\Omega$ near $p$ and let the multitype $\mathcal{M}(p)=(1,m_1,\ldots,m_n)$ be finite. Moreover, since $\Omega$ is pseudoconvex, the integers $m_1,\ldots,m_n$ are all even. 
 
 By the definition of multitype, there are distinguished coordinates $(z,w)=(z_1,\ldots,z_n,w)$ such that $p=0$ and $\rho(z,w)$ can be expanded near $0$ as follows:
$$
\rho(z,w)=\mathrm{Re}(w)+P(z)+R(z,w),
$$ 
 where $P$ is a $(1/m_1,\ldots,1/m_n)$-homogeneous plurisubharmonic polynomial that contains no pluriharmonic terms, $R$ is smooth and satisfies 
 $$
 |R(z,w)|\leq C \left( |w|+ \sum_{j=1}^n |z_j|^{m_j} \right)^\gamma,
 $$ 
 for some constant $\gamma>1$ and $C>0$. 
 
 In what follows, the weight of any multiindex $K=(k_1,\ldots,k_n)\in \mathbb N^n$ with respect to $\Lambda=(1/m_1,\ldots,1/m_n)$ is given by
   $$ wt(K)=\sum_{j=1}^n \frac{k_j}{m_j}.$$    
We note that  $wt(K+L)=wt(K)+wt(L)$ for any $K, L\in \mathbb N^n$. In addition, $\lesssim$ and $\gtrsim$ denote inequality up to a positive constant. Moreover, we will use $\approx $ for the combination of $\lesssim$ and $\gtrsim$. 
\begin{define}\label{def-order} We call $M_P=\{(z,w)\in \mathbb C^n\times \mathbb C\colon \mathrm{Re}(w)+P(z)<0\}$ an \emph{associated model} for $(\Omega,p)$. If the pointed domain $(\Omega,p)$ has an $h$-extendible  associated model, we say that $(\Omega,p)$ is \emph{$h$-extendible}. In this circumstance, we say that a sequence $\{\eta_j=(\alpha_j,\beta_j)\}\subset  \Omega$ \emph{converges $\Lambda$-nontangentially to $p$} if $|\mathrm{Im}(\beta_j)|\lesssim |\mathrm{dist}(\eta_j,\partial \Omega)|$ and $ \sigma(\alpha_j) \lesssim |\mathrm{dist}(\eta_j,\partial \Omega)|$, where 
$$
\sigma(z)=\sum_{k=1}^n |z_k|^{m_k}.
$$ 
Here and in what follows, $\mathrm{dist}(z,\partial\Omega)$ denotes the Euclidean distance from $z$ to $\partial\Omega$.
\end{define}
\begin{remark} It is well-known that $\{\eta_j\}\subset \Omega$ converges nontangentially to $p$ if $|\mathrm{Im}(\beta_j)|\lesssim |\mathrm{dist}(\eta_j,\partial \Omega)|$ and $|\alpha_{j k}|\lesssim|\mathrm{dist}(\eta_j,\partial \Omega)|$ for every $1\leq k\leq n$, where $\alpha_j=(\alpha_{j 1},\ldots,\alpha_{j n}$). Nevertheless, such sequence converges $\Lambda$-nontangentially to $p$ if $|\mathrm{Im}(\beta_j)|\lesssim |\mathrm{dist}(\eta_j,\partial \Omega)|$ and $|\alpha_{j k}|^{m_j}\lesssim|\mathrm{dist}(\eta_j,\partial \Omega)|$ for every $1\leq k\leq n$.
\end{remark}

We also need the following definition (cf. \cite{Yu95}).
\begin{define} Let $\Lambda=(\lambda_1,\ldots,\lambda_n)$ be a fixed $n$-tuple of positive numbers and $\mu>0$. We denote by $\mathcal{O}(\mu,\Lambda)$ the set of smooth functions $f$ defined near the origin of $\mathbb C^n$ such that
$$
D^\alpha \overline{D}^\beta f(0)=0~\text{whenever}~ \sum_{j=1}^n (\alpha_j+\beta_j)\lambda_j \leq \mu.
$$
If $n=1$ and $\Lambda = (1)$ then we use $\mathcal{O}(\mu)$ to denote the functions vanishing to order at least $\mu$ at the origin.
\end{define}

Now let us recall the following proposition, whose proof easily follows from the Taylor expansion (see \cite[Proposition $4.9$]{Yu95}).
\begin{proposition}\label{weight-small}
\begin{itemize}
\item[(i)] If $f\in \mathcal{O}(\mu,\Lambda)$ then $\frac{\partial f}{\partial z_j}$ and $\frac{\partial f}{\partial \bar z_j}$ are in $\mathcal{O}(\mu-\lambda_j,\Lambda)$ for $j=1,\ldots,n$.
\item[(ii)] Suppose that $f_i,~1\leq i\leq N$, are functions with $f_i\in \mathcal{O}(\mu_i,\Lambda)$. Then 
$$
\prod_{i=1}^N f_i \in \mathcal{O}(\mu,\Lambda), ~\text{where}~\mu=\sum_{i=1}^N \mu_i.
$$
\item[(iii)] If $f\in \mathcal{O}(\mu,\Lambda)$, then there are constants $C,\delta>0$ such that
$|f(z)|\leq C(\sigma_\Lambda(z))^{\mu+\delta}$ for all $z$ in a small neighborhood of $0$.

\end{itemize}
\end{proposition} 
By Proposition \ref{weight-small}, one easily obtains the following corollary.
\begin{corollary} \label{tangent-2} If $f\in \mathcal{O}(\mu,\Lambda)$, then there are constants $C,\delta>0$ such that
$| D^{p} \overline{D}^q f(z)|\leq C(\sigma_\Lambda(z))^{\mu-wt(p)-wt(q)+\delta}$ for every multi-indices $p,q\in \mathbb N^n$ with $wt(p)+wt(q)<\mu$ and for all $z$ in a small neighborhood of $0$.
\end{corollary}

\section{Proof of Theorem \ref{maintheorem1}}

This section is devoted to a proof of Theorem \ref{maintheorem1}. Throughout this section, the domain $\Omega$ and the boundary point $\xi_0\in \partial \Omega $ are assumed satisfy the hypothesis of Theorem \ref{maintheorem1}. Let $\rho$ be a local defining function for $\Omega$ near $\xi$ and let the multitype $\mathcal{M}(p)=(1,m_1,\ldots,m_n)$ be finite. Especially, because of the pseudoconvexity of $\Omega$, the integers $m_1,\ldots,m_n$ are all even. Let us denote by $\Lambda=(1/m_1,\ldots,1/m_n)$. By the definition of multitype, there are distinguished coordinates $(\tilde z,\tilde w)=(\tilde z_1,\ldots,\tilde z_n,\tilde w)$ such that $\xi_0=0$ and $\rho(\tilde z,\tilde w)$ can be expanded near $0$ as follows:
$$
\rho(\tilde z,\tilde w)=\mathrm{Re}(\tilde w)+P(\tilde z)+Q(\tilde z,\tilde w),
$$ 
 where $P$ is a $\Lambda$-homogeneous plurisubharmonic polynomial that contains no pluriharmonic monomials, $Q$ is smooth and satisfies 
 $$
 |Q(\tilde z,\tilde w)|\leq C \left( |\tilde w|+ \sum_{j=1}^n |\tilde z_j|^{m_j} \right)^\gamma,
 $$ 
 for some constant $\gamma>1$ and $C>0$. 
 
 By hypothesis of Theorem \ref{maintheorem1}, there exist a sequence $\{\varphi_j\}\subset \mathrm{Aut}(\Omega)$ and a point $a\in \Omega$ such that 
$\eta_j:=\varphi_j(a)$ converges $\Lambda$-nontangentially to $\xi_0$. Let us write $\eta_j=(\alpha_j,\beta_j)=(\alpha_{j1},\ldots,\alpha_{jn},\beta_j)$. Then one has 
\begin{itemize}
\item[(a)] $|\mathrm{Im}(\beta_j)|\lesssim |\mathrm{dist}(\eta_j,\partial \Omega)|$;
\item[(b)] $|\alpha_{jk}|^{m_k}\lesssim |\mathrm{dist}(\eta_j,\partial \Omega)|$ for $1\leq k\leq n$.
\end{itemize}

By following the proofs of Lemmas $4.10$, $4.11$ in \cite{Yu95}, after a change of variables
\[\begin{cases}
z=\tilde z;\\
 w=\tilde w+ b_1(\tilde z)\tilde w+b_2(\tilde z)\tilde w^2+b_3(\tilde z),
\end{cases}
\]
where $b_1,b_2, b_3$ are smooth functions of $\tilde z$ satisfying $b_j=O(|\tilde z|^2)$, $j=1,2,3$, there are local holomorphic coordinates $(z,w)$ in which $\xi_0=0$ and $\Omega$ can be described near $0$ as follows: 
 $$
 \Omega=\left\{\rho(z,w)=\mathrm{Re}(w)+ P(z) +R_1(z) + R_2(\mathrm{Im} w)+(\mathrm{Im} w) R(z)<0\right\}.
 $$ 
 Here $P$ is a $\Lambda$-homogeneous plurisubharmonic real-valued polynomial containing no pluriharmonic terms, $R_1\in \mathcal{O}(1, \Lambda),R\in \mathcal{O}(1/2, \Lambda) $, and $R_2\in \mathcal{O}(2)$. We would like to emphasize that in the new coordinates the sequence $\{\eta_j\}$ still has the properties $\mathrm{(a)}$ and $\mathrm{(b)}$. 

For any sequence $\{\eta_j=(\alpha_j,\beta_j)\}$ of points converging $\Lambda$-nontangentially to the origin in $U_0\cap\{\rho<0\}=:U_0^-$, we associate with a sequence of points $\eta_j'=(\alpha_{1j}, \cdots, \alpha_{nj}, a_j +\epsilon_j+i b_j)$, where $\epsilon_j>0$ and $\beta_j=a_j+i b_j$, such that $\eta_j'=(\alpha_j',\beta_j')$ is in the hypersurface $\{\rho=0\}$ for every $j\in\mathbb N^*$. We note that $\epsilon_j\approx \mathrm{dist}(\eta_j,\partial \Omega)$. Now let us consider the sequences of dilations $\Delta^{\epsilon_j}$ and translations $L_{\eta_j'}$, defined respectively by
$$
\Delta^{\epsilon_j}(z_1,\ldots,z_n,w)=\left(\frac{z_1}{\epsilon_j^{1/m_1}},\ldots,\frac{z_n}{\epsilon_j^{1/m_n}},\frac{w}{\epsilon_j}\right)
$$
and
$$
L_{\eta_j}(z,w)=(z,w)-\eta_j=(z-\alpha_j,w-\beta_j).
$$
Under the change of variables $(\tilde z,\tilde w):=\Delta^{\epsilon_j}\circ L_{\eta_j}(z,w)$, i.e.,
\[
\begin{cases}
w-\beta_j= \epsilon_j\tilde{w}\\
z_k-\alpha_{j k}=\epsilon_j^{1/m_k}\tilde{z}_k,\, k=1,\ldots,n,
\end{cases}
\]
one sees that $\Delta^{\epsilon_j}\circ L_{\eta_j'}(\alpha_j,\beta_j)=(0,\cdots,0,-1)$ for every $j\in \mathbb N^*$. Moreover, by using Taylor's theorem, the hypersurface $\Delta^{\epsilon_j}\circ L_{\eta_j'}(\{\rho=0\}) $ is defined by an equation of the form
\begin{align*}
\begin{split}
&0=\epsilon_j^{-1}\rho\left (L_{\eta_j}^{-1}\circ \left(  \Delta^{\epsilon_j} \right )^{-1}(\tilde z,\tilde w)\right)\\
&= \mathrm{Re} (\tilde w)+ R_2'(b_j) \mathrm{Im}(\tilde w) +  \mathrm{Im}(\tilde w) R(\alpha_j)+\epsilon_j^{-1}o(\epsilon_j)+P(\tilde z)\\
&+2\mathrm{Re}\sum\limits_{\substack{|p|>0\\ wt(p)\leq 1}} \frac{D^pP(\alpha_j)}{p!} \epsilon_j^{wt(p)-1} (\tilde z)^p+ \sum\limits_{\substack{|p|,|q|>0\\ wt(p+q)<1}}\frac{D^{p} \overline{D}^q P(\alpha_j)}{p!q!} \epsilon_j^{wt(p+q)-1} (\tilde z)^p (\overline{\tilde z})^q\\
&+2\mathrm{Re}\sum\limits_{\substack{|p|>0\\ wt(p)\leq 1}} \frac{D^p R_1(\alpha_j)}{p!} \epsilon_j^{wt(p)-1} (\tilde z)^p
+ \sum\limits_{\substack{|p|,|q|>0\\ wt(p+q)\leq1}} \frac{D^{p} \overline{D}^q R_1(\alpha)}{p!q!} \epsilon_j^{wt(p+q)-1} (\tilde z)^p (\overline{\tilde z})^q\\
&+ \epsilon_j^{-1} b_j \Big(2\mathrm{Re}\sum\limits_{\substack{|p|>0\\ wt(p)\leq 1}} \frac{D^p R(\alpha_j)}{p!} \epsilon_j^{wt(p)} (\tilde z)^p
+ \sum\limits_{\substack{|p|,|q|>0\\ wt(p+q)\leq1}} \frac{D^{p} \overline{D}^q R(\alpha_j)}{p!q!} \epsilon_j^{wt(p+q)} (\tilde z)^p (\overline{\tilde z})^q\Big).\\
\end{split}
\end{align*}

Since $\{(\alpha_j,\beta_j)\}_j$ is a sequence of points converging $\Lambda$-nontangentially to the origin in $U_0^-$, without loss of generality, we may assume that 
$$
\lim\limits_{j\to \infty}\pi_{1/\epsilon_j}(\alpha_j)=\alpha\in \mathbb C^n,
$$
where $\pi_t(z)=(t^{1/m_1} z_1,\ldots, t^{1/m_n} z_n)$ for $t\geq 0$. Hence, by Proposition \ref{weight-small} and Corollary \ref{tangent-2} one has
\begin{itemize}
\item[(i)] $\lim\limits_{j\to \infty}\frac{D^pP(\alpha_j)}{p!} \epsilon_j^{wt(p)-1}=\lim\limits_{j\to \infty}\frac{D^pP(\pi_{1/\epsilon_j}(\alpha_j))}{p!}= \frac{D^pP(\alpha)}{p!}$;
\item[(i)] $\lim\limits_{j\to \infty}\frac{D^p R_1(\alpha_j)}{p!} \epsilon_j^{wt(p)-1}=\lim\limits_{j\to \infty}\frac{D^p R(\alpha_j)}{p!} \epsilon_j^{wt(p)}=0 $ whenever  $wt(p)\leq 1$;
\item[(ii)] $\lim\limits_{j\to \infty} \frac{D^{p} \overline{D}^q P(\alpha_j)}{p!q!} \epsilon_j^{wt(p+q)-1}=\lim\limits_{j\to \infty} \frac{D^{p} \overline{D}^q P(\pi_{1/\epsilon_j}(\alpha_j))}{p!q!}= \lim\limits_{j\to \infty} \frac{D^{p} \overline{D}^q P(\alpha)}{p!q!}$ whenever $wt(p+q)< 1$;
\item[(iii)]  $\lim\limits_{j\to \infty} \frac{D^{p} \overline{D}^q R_1(\alpha_j)}{p!q!} \epsilon_j^{wt(p+q)-1}=\lim\limits_{j\to \infty} \frac{D^{p} \overline{D}^q R(\alpha_j)}{p!q!} \epsilon_j^{wt(p+q)}=0$ whenever $wt(p)+wt(q)\leq 1$;
\item[(iv)] $\lim\limits_{j\to \infty} R_2'(b_j) =\lim\limits_{j\to \infty} R(\alpha_j)= 0$.
\end{itemize}

Therefore, after taking a subsequence if necessary, we may assume that the sequence of domains $\Omega_j:=\Delta^{\epsilon_j}\circ L_{\eta_j'}(U_0^-) $ converges normally to the following model
$$
M_{P,\alpha}:=\left \{(\tilde z,\tilde w)\in \mathbb C^n\times\mathbb C\colon \mathrm{Re}(\tilde w)+P(\tilde z+\alpha)-P(\alpha)<0\right\},
$$
which is obviously biholomorphically equivalent to the model $M_P$.

Without loss of generality, in what follows we always assume that $\{\Omega_j\}$ converges to $M_P$. 

Now we need the following lemma which precises \cite[Lemme de localisation]{Ber95} (see also \cite[Lemma $2.1.1$]{Ga99}).
\begin{lemma}[Localization lemma]\label{local} Let $D$ be a domain in $\mathbb C^n$ and $\zeta_0\in \partial D$. Suppose that there exists a function $\varphi$ which is continuous on $\overline{D}\cap \{|z-\zeta_0|\leq R\}$ such that

\n 	
(i) $\varphi$ is plurisubharmonic on $D\cap \{|z-\zeta_0|< R\}$.

\n 
(ii) $\varphi>0$ on $\overline{D}\cap \{|z-\zeta_0|\leq r\}~(r<R)$.

\n
(iii) $\varphi<0$ on $\overline{D}\cap \{r'\leq |z-\zeta_0|\leq R'\}~(r<r'<R'<R)$.

\n 
Let $U:= D\cap \{|z-\zeta_0|< \fr{r}6\}, V:= D\cap \{|z-\zeta_0|< \fr{r}5\}.$
Then, there exists a constant $\tau_0\in (0,1)$ such that every holomorphic maps $f\colon \B^k \to D,$ where $\B^k$ is the unit ball in $\mathbb C^k,$ satisfies
$$f(0)\in U \, \Rightarrow\, f(\B^k (0, \tau_0)) \subset V,$$
where $\B^k (a,\tau_0):=\{z\in \mathbb C^k\colon |z-a|<\tau_0\}$ is the open ball of radius $\tau_0$ with center at $a$.
\end{lemma}
\begin{proof} We follow closely the proof of the localization lemma given in \cite{Ber95}, which in turns is based on Theorem 3 in \cite{Si81}.
Using a patching technique as in \cite{Ber95}, we can construct a bounded negative \psh\ function $\tilde \va$ on $D$ such that $\tilde \va-\vert z\vert^2$
is \psh\ on $D\cap \{|z-\zeta_0|< r\}.$ Then, by an ingenious argument using the maximum principle we obtain the following lower bound for the infinitesimal Kobayashi metric
$$F_D (z,v) \ge \sqrt{\frac{2}{r}} e^{\frac{M}2 \tilde \va (z)} \Vert v\Vert, \forall v \in \mathbb C^n, \forall z \in D\cap \{|z-\zeta_0|< \fr{r}4\}.$$
Now suppose the lemma is false, then there exists a sequence of holomorphic maps $f_j: \B^k \to D$ and $a_j \to 0, a_j \in \B^k$ with 
$f_j (0) \in V$ but $f_j (a_j) \not\in U.$
By the decreasing property of the Kobayashi pseudo-distance we obtain 
$$d_D (f_j(0), f_j (a_j)) \le d_{\B^k} (0, a_j) \to 0 \ \text{as}\ j \to \infty.$$
On the other hand, we can find $b_j \in D \cap \{|z-\zeta_0|= \fr{r}5\}$ such that
$$d_D (f_j(0), f_j (a_j))+\frac{1}{j} \ge d_D (f_j(0), b_j).$$
For a real smooth curve $\gamma \subset D$ joining $f_j (0)$ and $b_j$ we have
$$k_D (f_j (0), b_j) \ge \int_0^1 F_D (\gamma (t), \gamma' (t)) 
\ge  \sqrt{\frac{2}{r}} e^{\frac{M}2 \inf\limits_{z \in D} \tilde \va (z)} \Vert f_j (0) -b_j \Vert.$$
It follows that 
$\varliminf\limits_{j \to \infty} k_D (f_j (0), b_j) >0.$ 
Putting all these estimates together we obtain a contradiction.
\end{proof}	
We need the following technical lemma which plays a key role in the proof of Theorem \ref{maintheorem1}.
\begin{lemma}\label{Ga99}  Let $\{\Omega_j\}$ be a sequence of domains in $\mathbb C^{n+1}$ converging to $M_P$. Let 
$K$ be a compact subset of $M_P.$ Then there exists a compact subset $L$ of $M_P$, an index $j(K) \ge 1$, 
and $\tau \in (0,1)$ having the following properties:
If $g: \B^k  \to \Omega_j$ is holomorphic for $j \ge j(K)$ and $g(0)\in K$ then $g(\B^k (0,\tau)) \subset L.$
\end{lemma}
\begin{proof} We split the proof into two steps.
	
\n {\it Step 1.} We show that there exist neighborhoods $\widetilde U, \widetilde U'$ of the origin and $\tau_0>0$
such that:  {\it For $j$ large enough, if $f\colon \B^k \to \Omega_j$ is holomorphic and 
$f(0)\in \widetilde U'$ then $f(\B^k_{\tau_0}) \subset \widetilde U.$}
For this purpose, we note that
there exists a plurisubharmonic peak function for $M_P$ at $(0',0)$ (see \cite{Yu94}). Thus we may find 
$0<r<r'<R'<R,$ a \psh\ function $\va$ on $M_P$ which is continuous on $\overline{M_P}$ such that $\va>0$ on $M_P \cap \{\vert z\vert<r\}$ and $\va<0$ on  $M_P \cap \{r'<\vert z\vert<R'\}.$

By setting $\ve_0:=\fr{r}{7}$, since the sequence $\{\Omega_j\}$ converges to $M_P$ as $j\to \infty,$ we can find $j_0 \ge 1$ and a large open ball $B_r$ around $\xi_0:=(0,\ve_0)$ such that for $j \ge j_0$ we have
$$\Omega_j \subset \widetilde \Omega_r:=M_{P,r} \cup (\mathbb C^{n+1} \setminus \overline {B_r}),$$
where $M_{P,r}:=\{(z,w): \re(w)+P(z)<\ve_0\}.$
Now consider the following \nhd s of $(0,0)$
$$\widetilde U:= \{|z-\xi_0|<\fr{r}5\}, \widetilde U':=\{\vert z -\xi_0\vert <\fr{r}6\}.$$
By applying Lemma 1 to $\widetilde \Omega_r,$ the peaking function $\psi (z,w):=\va(z, w-\ve_0)$ and the datum $r',r,R',R$ we obtain $\tau_0>0$ satisfying the conclusion of \emph{Step 1}.

\n 
{\it Step 2.} We argue by contradiction. If the lemma is false then we can find a sequence
$\B^k \ni \xi_j \to 0,$ holomorphic maps $g_j: \B^k \to \Omega_j$ such that 
\begin{equation} \label{conv}
g_j (0) \in K \subset M_P \ \text{but}\  g_j (\xi_j) \to \partial M_P \cup \{\infty\}.
\end{equation}
The key step in deriving a contradiction is to show that $\{g_j\}$ is locally uniformly near the origin.
For this, choose $\lambda_0>0$ so big that 
$\Delta^{\lambda_0} (K) \subset \widetilde U'.$
Then by \emph{Step 1} we obtain 
$$(\Delta^{\lambda_0}\circ g_j)(\B^k_{\tau_0}) \subset \widetilde U,  \forall j.$$
Hence for every $j$ we have $g_j(\B^k_{\tau_0}) \subset \left(\Delta^{\lambda_0}\right)^{-1}(\widetilde U),$ a bounded open subset of 
$\mathbb C^{n+1}$.
Now, by Montel's theorem, after passing to a subsequence we may assume that $g_j$ converges uniformly on compact sets of 
$\B^k_{\tau_0}$ to a holomorphic map $g:\B^k_{\tau_0} \to \mathbb C^{n+1}$. It follows that 
$$\lim\limits_{j \to \infty} g_j (0)=g(0)=\lim\limits_{j \to \infty} g_j(\xi_j).$$ 
We obtain a contradiction to (\ref{conv}). Hence we get a constant $\tau>0$ that satisfies both conditions in \emph{Step 1} and \emph{Step 2}.
\end{proof} 	
\n
The main step in the proof of Theorem 1 is included in the following result. We also use this proposition crucially in the 
next section.
\begin{proposition}\label{pro-scaling} Let $\omega$ be a domain in $\mathbb C^k$, $a\in \omega$ and $\sigma_j: \omega \to \Omega_j$ be a sequence of holomorphic mappings such that $\{\sigma_j(a)\}\Subset M_P$. Then $\{\sigma_j\}$ contains a subsequence that converges locally uniformly to a holomorphic map $\sigma: \omega \to M_P$. 
\end{proposition}
\begin{proof}
Choose $r>0$ so small such that $\B^k (a,r) \Subset \omega.$ Set 
$$g_{a,j}(z):= \sigma_j \Big (r (z+\fr{a}r) \Big ) \ j \ge 1.$$
Then $g_{a,j}: \B^k \to \Omega_j$ and satisfies $g_{a,j} (0)=\sigma_j (a)$ is contained in a fixed compact subset $K$ of $M_P$.
It follows, in view of Lemma 3, that $\sigma_j (\B^k (a, \tau r))$ is included in some compact subset $L$ of $M_P$ for $j$ large enough.
Now we let $\omega'$ be the collection of $x \in \omega$ such that there exists a \nhd\ $U$ of $x$ such that $\sigma_j (U)$ 
is contained in a compact subset of $M_P$ for all $j$ {\it large enough}.
Then $\omega'$ is an open subset of $\omega$ and $a \in \omega'.$ 
We claim that $\omega'=\omega.$ If this is not so, then we can find a point $x_0 \in \omega \cap \partial \omega'.$
Choose $x_1 \in \omega'$ closed to $x_0$ and $r'>0$ so small that: 
$$x_0 \in \B^k (x_1, \tau r')\subset \B^k (x_1,r') \Subset \omega.$$
By considering the new sequence 
$$\sigma'_j (z) = \sigma_j \Big (r' (z+\frac{x_1}{r'})\Big ), \ z \in \B^k.$$
We may apply Lemma 3 again to infer that $\sigma_j (\B^k (x_1, \tau r'))$ is contained in some compact set of $M_P$ for $j$ large enough. This implies that $x_0 \in \omega'.$ We reach a contradiction. Thus $\omega'=\omega$ as claimed. 

Finally, in view of Montel's theorem, after passing to a subsequence, we may assume that $\sigma_j$ uniformly converges on compact sets of $\omega$ to a holomorphic map $\sigma: \omega \to \mathbb C^n.$ By the above reasoning we see that $\sigma (\omega) \subset M_P.$ The desired conclusion follows. 
\end{proof}

We are now ready to give a proof of Theorem \ref{maintheorem1}.
\begin{proof}[Proof of Theorem \ref{maintheorem1}]
Assume that $(\Omega,\xi_0)$ is $h$-extendible. It means that the model $M_P$ is also $h$-extendible.
By the hypothesis, the sequence $\{\eta_j:=\varphi_j(a)\}$ converges $\Lambda$-nontangentially to  $\xi_0=(0',0)$. Then one can find a sequence $\{\epsilon_j\}\subset \mathbb R^+$ converging to $0^+$ such that the sequence of points $\eta_j'=\eta_j+(0',\epsilon_j)$ is in the hypersurface $\{\rho=0\}$ for every $j\ge 1$.
Let us define $T_j:=\Delta^{\epsilon_j}\circ L_{\eta_j'}$ and $\sigma_j:=T_j\circ \varphi_j\colon \varphi_j^{-1}(U_0^-)\to \Omega_j.$
Then one sees that $T_j(\eta_j)=(0',-1)$ 
and $\{\sigma_j\}$ is a sequence of biholomorphic mappings satisfying 
$$
\sigma_j(a)=b:=(0',-1),\; j\geq 1.
$$
Thus, by Proposition \ref{pro-scaling}, after passing to a subsequence, we may assume that $\sigma_j$ converges locally uniformly to a holomorphic map $\sigma: \Omega \to M_P$ which satisfies $\sigma (a)=b$. 

On the other hand, since $\Omega$ is taut, the sequence $\sigma_j^{-1}\colon \Omega_j\to \varphi_j^{-1}(U_0^-)\subset \Omega$ is also normal. Since $\sigma_j^{-1} (b)=a \in \Omega$, we may also assume, after switching a subsequence that
$\sigma_j^{-1}$ converges locally uniformly to a holomorphic map $\sigma^*: M_P \to \Omega$. It then follows from Proposition \ref{T:7} that $\sigma^*$ is the inverse of $\sigma$ and so $\sigma$ maps
$\Omega$ biholomorphically onto $M_P$. It is then obvious that $\sigma (a)=\lim\limits_{j \to \infty} \sigma_j (a)=(0',-1).$ Thus, we have shown the assertion (a). 

For (b), we claim that there exists a sequence $\xi_j \to \xi_0$ such that
$$\varliminf\limits_{x \to \xi_j} \vert \sigma (x) \vert <\infty \ \forall j.$$
If the claim fails then we may find an open ball $B$ around $\xi_0$ such that
$$\lim\limits_{x \to \xi} \vert \sigma (x) \vert =\infty \ \forall \xi \in B \cap \partial \Omega.$$
Then we choose a {\it bounded} holomorphic function $f$ on $M_P$ such that $f \not \equiv 0$ and
$$\lim\limits_{\vert z\vert \to \infty, z \in M_P} f(z)=0.$$ 
Indeed, it suffices to take $N=1$ in the proof of Theorem $3.4$ in [Yu94] to obtain the desired function $f$.
It follows that $\hat f:=f \circ \sigma$ is bounded holomorphic on $\Omega$ and satisfies 
$$\lim\limits_{x \to \xi} \hat f (x)=0 \ \forall \xi \in B \cap \partial \Omega.$$
Suppose that $\hat f \not\equiv 0$ on $\Omega$. Then $S:=\{x \in \Omega: \hat f(x)=0\}$ is a complex hypersurface of $\Omega.$
Thus we can find a point $x_0 \in \Omega \setminus S$ that is so close to $\partial \Omega$ such that for 
some $\xi^0 \in B \cap \partial \Omega$ the open segment connecting $\xi^0$ and $x_0$ stays in $\Omega.$ 
Let $l$ be the complex line joining $x_0$ and $\xi^0$ and $\Omega_l$ be the connected component of 
$l \cap \Omega$ that contains $x_0.$ Then $\hat f|_l$ is a bounded holomorphic function on $\Omega_l$ that tends to $0$
at an open piece of $\partial \Omega_l.$
By applying the two constant theorem to the bounded subharmonic function $\log \vert \hat f|_l\vert$ we infer that 
$\log \vert \hat f|_l \vert$ must be identically $-\infty$ on $\Omega_l.$ In particular $\hat f (x_0)=0$, which is absurd.
Hence $\hat f \equiv 0$ on $\Omega$, which is impossible since $\sigma$ is biholomorphic. 
Thus our claim is valid. 

On the other hand, since $\Omega$ is of finite type at $\xi_0,$ we may achieve that $\Omega$ is of finite type at every point $\xi_j.$ Furthermore, one can also find sequences $\Omega \ni \{x_{k,j}\} \to \xi_j$ such that
$\sigma (x_{k,j}) \to \tilde \xi_j \in \partial M_P$ as $k \to \infty.$
Now we can apply Proposition 3 in [Ber95] to reach the conclusion (b). The proof is thereby complete.
\end{proof}
\section{ Proof of Theorem \ref{theorem 1.2}}
 
Throughout this section, let $\Omega$ be a domain  and $\xi_0\in \partial \Omega $ be as in the hypothesis of Theorem \ref{theorem 1.2}. Let $\rho$ be a local smooth defining function for  $\Omega $ near $\xi_0$. After a  change of coordinates, we can find the coordinate functions $(z_1,\ldots, z_n,w)$ defined on a neighborhood $U_0$ of $\xi_0$ such that $\xi_0=0$ and
$\Omega$ can be described locally near $0$ as 
\begin{equation}\label{eq77}
\Omega=\left\{\rho(z,w)=\mathrm{Re}(w)+ P(z) +R_1(z) + R_2(\mathrm{Im} w)+(\mathrm{Im} w) R(z)<0\right\}.
\end{equation}
Here $P$ is a $\Lambda$-homogeneous plurisubharmonic real-valued polynomial containing no pluriharmonic monomials, $R_1\in \mathcal{O}(1, \Lambda),R\in \mathcal{O}(1/2, \Lambda) $, and $R_2\in \mathcal{O}(2)$. 
Let us fix a small neighborhood $U_0$ of $0$ and consider any point $\eta=(\alpha,\beta)\in U_0$.
Now we define an anisotropic dilation $\Delta^\epsilon$ and a translation $L_\eta$, respectively, by
$$
\Delta^\epsilon(z_1,\ldots,z_n,w)=\left(\frac{z_1}{\epsilon^{1/m_1}},\ldots,\frac{z_n}{\epsilon^{1/m_n}},\frac{w}{\epsilon}\right)
$$
and
$$
L_\eta(z,w)=(z,w)-\eta=(z-\alpha,w-\beta).
$$
Let $\{\eta_j\}$ be a sequence in $\Omega$ converging $\Lambda$-nontangentially to $\xi_0=0$. Without loss of generality, we may assume that $\eta_j=(\alpha_j,\beta_j)\in U_0^-:=U_0\cap\{\rho<0\}$ for all $j$. For this sequence $\{\eta_j\}$, one associates with a sequence of points $\eta_j'=(\alpha_{1j}, \ldots, \alpha_{n_j},\beta_{j}+\epsilon_j)$, $ \epsilon_j>0$, $\eta_j'$ in the hypersurface $\{\rho=0\}$. Let us consider the sequences of dilations $\Delta^{\epsilon_j}$ and translations $L_{\eta_{j}'}$. Then $\Delta^{\epsilon_j}\circ L_{\eta_{j}'}({\eta}_j)=(0,\ldots,0,-1)$ and moreover, by Lemma 1, after taking a subsequence, one can deduce that $\Delta^{\epsilon_j}\circ L_{\eta_{j}'}(U_0^-)$ converges to the following model
\begin{equation*}
M_P:=\left\{\hat\rho:=\mathrm{Re}(w)+ P(z)<0\right\},
\end{equation*}
where $P(z)$ is the real $\Lambda$-homogeneous polynomial given in (\ref{eq77}).

Now we are ready to give a proof of Theorem \ref{theorem 1.2}. To do this, let us set $\delta _j=2(1-s_{\Omega}(\eta_j))$ for all $j$. Then by our assumption, for each $j$ there exists an injective holomorphic map $f_j:\Omega\to \mathbb{B}^{n+1} $ such that $f_j(\eta_j)=(0,\ldots,0)$ and $\mathbb{B}^{n+1}(0;1-\delta _j)\subset f_j(\Omega)$. 
By Proposition \ref{T:8}, one sees that $f_j(\Omega \cap U_0)$ converges to $\mathbb B^{n+1}$. So, Proposition \ref{pro-scaling} shows that the sequence $T_j\circ f_j^{-1} \colon f_j(\Omega \cap U_0) \to  T_j(\Omega \cap U_0) $ is normal and its limits are holomorphic mappings from $\mathbb B^{n+1}$ to $M_P$, where $T_j:=\Delta^{\epsilon_j}\circ L_{\eta_j'}$ for every $j\in \mathbb N^*$. Moreover, by Montel's theorem the sequence $ f_j\circ T_j^{-1} \colon T_j(\Omega \cap U_0)\to   f_j(\Omega \cap U_0) \subset \mathbb B^{n+1}$ is also normal. We note that since $T_j\circ f_j^{-1}(0)=(0',-1)\in M_P$, it follows that the sequence $T_j\circ f_j^{-1}$ is not compactly divergent. Therefore, by Proposition \ref{T:7}, after taking some subsequence we may assume that $T_j\circ f_j^{-1}$ converges uniformly on every compact subset of $\mathbb B^{n+1}$ to a biholomorphism from $\mathbb B^{n+1}$ onto $M_P$. 
 
Observe that the unit ball $\mathbb{B}^{n+1}$ is biholomorphic to the Siegel half-space 
$$\mathcal{U}:=\{(z,w)\in \mathbb{C}^n \colon\mathrm{Re}(w) +|z_1|^2+|z_2|^2+\cdots+|z_n|^2<0\}.$$ 
Hence, we may assume that there exists a biholomorphism $\psi\colon M_P\to \mathcal{U}$.

As in the end of the proof of Theorem 1, we can find 
a bounded holomorphic function $\phi$ on $\mathcal{U}$ which is continuous on $\overline{\mathcal{U}}, \phi \not\equiv 0$ 
and tends to $0$ at infinity. (Actually in this concrete situation we may write down explicitly such a function $\phi$.) We claim that there exists $t_0\in \mathbb R$ such that $\varliminf\limits_{\substack{x\to 0\\
 x<0}}|\psi (0',x+it_0)|<+\infty$. Indeed, if this would not be the case, the function $\phi\circ\psi$ would equal to $0$ on the half-plane $\{\mathrm{Re}(w)<0, z=0\}$ and this is impossible since $\phi \not\equiv 0$. Therefore, we may assume that there exists a sequence $x_k<0$ such that $\lim x_k=0$ and $\lim\psi (0',x_k+it_0)=p_0\in \partial \mathcal U$. Hence, it is proved in \cite[Theorem $2.1$]{CP01} that under these circumstances $\psi$ extends holomorphically to a neighborhood of $(0',it_0)$. Since the Levi form is preserved under local biholomorphisms around a boundary point, it follows that $M_P$ is strongly pseudoconvex at $(0',it_0)\in \partial M_P$. This yields that $m_1=\cdots=m_n=2$ and $P(z)=|z_1|^2+\cdots +|z_n|^2$, and thus $\Omega$ is strongly pseudoconvex at $\xi_0$, as desired.  \hfill $\Box$

\begin{Acknowledgement}  Part of this work was done while the authors were visiting the Vietnam Institute for Advanced Study in Mathematics (VIASM) in 2019. We would like to thank the VIASM for financial support and hospitality. The first author was supported by the Vietnam National Foundation for Science and Technology Development (NAFOSTED) under Grant Number 101.02-2017.311. The second author was supported by the Vietnam National Foundation for Science and Technology Development (NAFOSTED) under Grant Number 101.02-2019.304. It is a pleasure to thank Hyeseon Kim for stimulating discussions. Especially, we would like to express our gratitude to the referee. His/her valuable comments on the first version of this paper led to significant improvements.
\end{Acknowledgement}

\end{document}